\documentclass[11pt]{amsart}

\usepackage{amsmath,amssymb,amsthm}
\usepackage{amsmath,amssymb,amsthm,amscd}
\usepackage{enumerate}
\usepackage[frame,cmtip,arrow,matrix,line,graph,curve]{xy}
\usepackage{graphpap, color}
\usepackage[mathscr]{eucal}
\usepackage{color}
\numberwithin{equation}{section}
\newtheorem{prop}{Proposition}
\newtheorem{theo}[prop]{Theorem}
\newtheorem{lemm}[prop]{Lemma}
\newtheorem{coro}[prop]{Corollary}

\newtheorem{rema}[prop]{Remark}
\newtheorem{defi}[prop]{Definition}

\theoremstyle{definition}

\newtheorem*{ack}{Acknowledgment}
\theoremstyle{remark}

\newcommand{\p}{\partial}

\newcommand{\ppr}{(\sqrt{-1}/2) \partial \bar{\partial}}
\newcommand{\ppfr}{\frac{\sqrt{-1}}{2} \partial \bar{\partial}}
\newcommand{\fr}{\frac{\sqrt{-1}}{2}}

\def\lab{\label}

\begin{document}

\title[semilinear equations]{Semilinear equations, the $\gamma_k$ function, and generalized Gauduchon metrics}

\author{Jixiang Fu}
\address{Institute of Mathematics\\ Fudan University \\ Shanghai
200433, China} \email{majxfu@fudan.edu.cn}
\author{Zhizhang Wang}
\address{Institute of Mathematics\\ Fudan University \\ Shanghai
200433, China} \email{youxiang163wang@163.com}
\author{Damin Wu}
\address{Department of Mathematics \\
         The Ohio State University \\
         1179 University Drive, Newark, OH 43055, U.S.A.}
\email{dwu@math.ohio-state.edu}

\begin{abstract}
In this paper, we generalize the Gauduchon metrics on a compact
complex manifold and define the $\gamma_k$ functions on the space of
its hermitian metrics.
\end{abstract}
\maketitle
\section{Introduction}

Let $X$ be a compact $n$-dimensional complex manifold. Let $g$ be a
hermitian metric on $X$ and  $\omega$  its hermitian form. It is
well known that if $d\omega=0$, then $g$ or $\omega$ is called a
K\"ahler metric and therefore $X$ is called a K\"ahler manifold.
When $X$ is a non-K\"ahler manifold, one can consider the other
conditions on $\omega$ such as
\begin{equation} \label{eq:bal}
  d\omega^k=0, \qquad  2\leq k\leq n-1.
\end{equation}
 If $d(\omega^{n-1})=0$, then $g$ or $\omega$ is called a balanced metric and so $X$ is called a balanced manifold \cite{Mic}.
 However, when $2\leq k \leq n-2$, $d\omega^k=0$ automatically yields $d\omega=0$ \cite{GH}.
 Instead of \eqref{eq:bal}, one can consider the $k$-K\"ahler condition \cite{AA}. A complex manifold is called
 $k$-K\"ahler if it admits a closed complex transverse $(k,k)$--form. By this definition,
 a complex manifold is 1-K\"ahler if and only if it is K\"ahler; it is $(n-1)$--K\"ahler if and only if it is
 balanced.

One can also generalize the K\"ahler condition along other directions, for instance,
\begin{equation}\lab{eq:pbpclosed}
\partial\bar\partial\omega^{k}=0, \qquad 1\leq k\leq n-1.
\end{equation}
When $k = n-1$, the metric $\omega$ is called a \emph{Gauduchon} metric.
Gauduchon~\cite{Gau} proved an interesting result that, for any hermitian metric
$\omega$ on a compact complex $n$-dimensional manifold $X$, there exists a
unique (up to a constant) smooth function $v$ such that
\begin{equation} \label{eq:Gaud}
    \partial\bar\partial(e^v\omega^{n-1})=0 \qquad \textup{on $X$}.
\end{equation}
Thus, the Gauduchon metric always exists on a compact complex
manifold. It is important in complex geometry since one can use such
a metric to define the degree, and then make sense of the stability
of holomorphic vector bundles over a non-K\"ahler complex manifold
(see \cite{LiYau}).

When $k = n-2$, the metric $\omega$ satisfying \eqref{eq:pbpclosed}
is called an \emph{astheno-K\"ahler} metric. Jost and Yau \cite{JY}
used this condition to study  hermitian harmonic maps, and extended
Siu's rigidity theorem to non-K\"ahler complex manifolds.

When $k = 1$, the metric $\omega$ in \eqref{eq:pbpclosed} is called
a \emph{pluriclosed} metric, which is also called strong KT
(K\"ahler with torsion) metric (see  \cite{GGP,FT} and the
references therein). Such a condition appeared in \cite{Dem,Bis} as
a technical condition. Recently, Streets and Tian \cite{ST}
introduced a hermitian Ricci flow under which the pluriclosed metric
is preserved.

It is important to find  specific hermitian metrics on non-K\"ahler
complex manifolds. J. Li, S.-T. Yau and Fu \cite{FLY} have
constructed balanced metrics on complex structures of manifolds
$\#_{k\geq 2}(S^3\times S^3)$ which are obtained from the conifold
transition of Calabi-Yau threefolds. As a corollary, there exists no
pluriclosed metric on such manifolds. We note here that the specific hermitian geometry of threefolds $\#_k(\mathbb{S}^3\times\mathbb{S}^3)$ was first considered by Bozhkov
\cite{Bo1,Bo2}. In this paper, we generalize
(\ref{eq:pbpclosed}) to weaker conditions:
\begin{equation}\lab{k-gauduchon}
 \partial\bar\partial \omega^{k}\wedge \omega^{n-k-1}=0,\ \ \ 1\leq k\leq n-1.
\end{equation}

\begin{defi} \label{de:kGau}
Let $\omega$ be a hermitian metric on an $n$-dimensional
complex manifold $X$, and $k$ be an integer such that $1 \le k \le n-1$. We call
$\omega$ the $k$-th Gauduchon metric if $\omega$ satisfies (\ref{k-gauduchon}).
\end{defi}
Note that an $(n-1)$-th Gauduchon metric is the classic Gauduchon
metric. The natural question is whether there exists any $k$-th
Gauduchon metric, $1 \le k \le n-2$, on a complex manifold.  To answer this question, one
way is to look for such a
metric in the conformal class of a given
hermitian metric $\omega$ on $X$:
\begin{equation} \label{eq:nGaud}
\partial\bar\partial(e^v\omega^k)\wedge \omega^{n-k-1}=0.
\end{equation}
However, equation \eqref{eq:nGaud} in general needs not admit a solution (see below for reasons). In this paper, we solve the equation
 \begin{equation} \label{eq:ksl2}
     \partial\bar\partial (e^v \omega^{k}) \wedge \omega^{n-k-1}  = \gamma_k e^v \omega^n
  \end{equation}
for some constant $\gamma_k$ satisfying the compatibility condition.
The constant $\gamma_k$, if nonzero, can be viewed as an obstruction
for the existence of a $k$-th Gauduchon metric in the conformal
class of $\omega$, for $1\le k < n-1$.


Equation \eqref{eq:ksl2} can be reformulated, in a slightly more
general form, as follows: Let $(X,\omega)$ be an $n$-dimensional
compact hermitian manifold, and $B$ be a smooth real $1$-form on
$X$. For any smooth function $f$ on $X$ satisfying
\begin{equation} \label{eq:cpbf}
        \int_X f \omega^n = 0,
\end{equation}
we consider the following semilinear equation
\begin{equation} \label{eq:sl0}
   \Delta v + |\nabla v|^2 + \langle B, dv \rangle = f \qquad \textup{on $X$}.
\end{equation}
Here $\Delta$ and $\nabla$ are, respectively, the Laplacian and covariant differentiation associated with $\omega$. 
Clearly, equation \eqref{eq:sl0} needs not have a solution, due to the compatibility condition \eqref{eq:cpbf}. For instance, let $\omega$ be 
 balanced and $B= 0$, then in order that \eqref{eq:sl0} has a solution the function $f$ has to be zero. Nonetheless, we shall show that, there is a smooth function $v$ so that equation~\eqref{eq:sl0} holds up to a unique constant $c$. More generally, we have the following result:
\begin{theo} \label{th:sl1}
  Let $(X,\omega)$ be a compact hermitian manifold, $B$ be a smooth real $1$-form on $X$, and $\psi \in C^{\infty}(\mathbb{R})$ satisfy
  \begin{equation} \label{eq:defpsi}
     \liminf_{t \to +\infty} \frac{\psi(t)}{t^{\mu}} \ge \nu > 0, \quad \textup{where $\mu > 1/2$ and $\nu$ are constants}.
  \end{equation}
  Then, for each $f \in C^{\infty}(X)$ satisfying \eqref{eq:cpbf}, there exists a unique constant $c$, and a smooth function $v$ on $X$, unique up to a constant, such that
  \begin{equation} \label{eq:sl1}
        \Delta v + \psi(|\nabla v|^2) + \langle B, dv \rangle = f + c \qquad \textup{on $X$}.
  \end{equation}
\end{theo}
\begin{rema} \label{re:cmpa}
  The compatibility condition of \eqref{eq:sl1} implies that
  \[
     c = \frac{\int_X (\Delta v + \psi(|\nabla v|^2) + \langle B, dv \rangle)\, \omega^n}{\int_X \omega^n},
  \]
  which in general is nonzero.
\end{rema}
Letting $\psi(t) = t$ on $\mathbb{R}$, we obtain an application of Theorem~\ref{th:sl1}:
\begin{coro} \label{co:sl2}
  Let $(X,\omega)$ be an $n$-dimensional compact hermitian manifold. For any integer $1 \le k \le n-1$, there exists a unique constant $\gamma_k$, and a function $v \in C^{\infty}(X)$ satisfying that
  \begin{equation} \label{eq:sl2}
     \ppr (e^v \omega^{k}) \wedge \omega^{n-k-1}  = \gamma_k e^v \omega^n.
  \end{equation}
  The solution $v$ of \eqref{eq:sl2} is unique up to a constant.
  In particular, when $k = n -1$ we have $\gamma_{n-1} = 0$. If $\omega$ is K\"ahler,
  then $\gamma_k = 0$ and $v$ is a constant, for each $1 \le k \le n-1$.
\end{coro}

\begin{rema}
When $k=n-1$, this corollary recovers the classical result of Gauduchon~\cite{Gau}.
\end{rema}

By Corollary~\ref{co:sl2}, we can associate each hermitian metric
$\omega$ a unique constant $\gamma_k(\omega)$. Clearly, $\gamma_k =
\gamma_k(\omega)$ is invariant under biholomorphisms. Furthermore,
we will prove that $\gamma_k$ depends smoothly on the hermitian
metric $\omega$ (see Proposition~\ref{pr:moduli}); and that
$\gamma_k(\omega) = 0$ if and only if there exists a  $k$-th
Gauduchon metric in the conformal class of $\omega$ (Proposition~\ref{pr:kGau}).

We will prove in Proposition~\ref{pr:cfinv} that the sign of $\gamma_k(\omega)$, denoted by $(\textup{sgn} \gamma_k)(\omega)$,
is invariant in the conformal class of $\omega$.
We denote by $\Xi_k(X)$ the range of $\text{sgn}\gamma_k$. By definition
$\Xi_k(X) \subset \{-1,0, 1\}$ for each $k$, and by Corollary~\ref{co:sl2} we have $\Xi_{n-1}(X)=\{0\}$.
A natural question is \emph{whether $\Xi_k(X)=\{-1,0, 1\}$ for any
$1\leq k\leq n-2$ on any compact complex manifold $X$}.
Indeed, if $\Xi_k(X) \supset \{-1, 1\}$ then the answer is positive,
by Proposition~\ref{pr:moduli}. Thus, there will be a $k$-th
Gauduchon metric on $X$. We can also ask whether $\Xi_k(X)$ is
invariant under the modification. These questions will be
systematically studied later. As a first step, we obtain the
following result.
\begin{theo} \label{th:posgam}
  For $n = 3$, we have $1\in \Xi_1(X)$. Namely, for any $3$-dimensional hermitian manifold $X$, there exists a hermitian metric $\omega$ such that $\gamma_1(\omega) > 0$. In particular, there is no 1-st Gauduchon metric in the conformal class of $\omega$.
\end{theo}
Then, we combine the above results to prove that, as an example,
$\Xi_1=\{-1,0,1\}$ on the three-dimensional complex manifolds
constructed by Calabi~\cite{Ca}. As a consequence, there exists a
1-st Gauduchon metric on these manifolds. It is well-known that such
manifolds are non-K\"ahler but admit balanced metrics. We do not
know whether there exists any pluriclosed metric on them.

Another example we considered is $Y = S^5 \times S^1$, endowed with a  complex structure so that the natural projection $\pi: S^5 \times S^1 \to \mathbb P^2$ is holomorphic. This would imply that there is no balanced metrics on $S^5 \times S^1$. Moreover, we can prove that $S^5 \times S^1$ does not admit any pluriclosed metric. On the other hand, by considering a natural hermitian metric on $S^5\times S^1$, we are able to show that $\Xi_1(S^5 \times S^1) = \{-1,0,1\}$. Thus, $S^5 \times S^1$ admits a 1-st Gauduchon metric.

We shall solve equation~\eqref{eq:sl1} by the continuity method. In
Section~\ref{se:NP}, we set up the machinery and prove the openness.
The closedness and \emph{a priori} estimates are established in
Section~\ref{se:sl}. In Section~\ref{se:UC}, we prove the uniqueness
part in Theorem~\ref{th:sl1} and also prove Corollary~\ref{co:sl2}.
In Section~\ref{se:GG}, we discuss the relation between $\gamma_k$
and the $k$-th Gauduchon metric. In section~\ref{se:posgam}, we
prove Theorem~\ref{th:posgam}, and explicitly construct a metric
with positive $\gamma_1$ on the complex $3$-torus.  As another
example, we show that the natural balanced metric on the Iwasawa
manifold has a positive $\gamma_1$ number. In section 7, we
establish the existence of $1$-st Gauduchon metric on Calabi's
$3$-dimensional non-K\"ahler manifold, by using
Theorem~\ref{th:posgam} and proving that the balanced metric on the
manifold has a negative $\gamma_1$ number. In the last section, we prove the existence of a 1-st Gauduchon metric on $S^5\times S^1$. We also show the nonexistence of balanced metric and pluriclosed metric on $S^5 \times S^1$.

\begin{ack}
The authors would like to thank Professor S.-T.~Yau for helpful
discussion. Part of the work was done while the third named author
was visiting Fudan University, he would like to thank their warm
hospitality. Fu is supported in part by NSFC grants and LMNS.
\end{ack}

\section{Notation and preliminaries} \label{se:NP}
Throughout this note, we use the following convention: We write
\[
   \omega = \fr \sum_{i,j = 1}^n g_{i\bar{j}} dz^i \wedge d \bar{z}^j.
\]
Let $(g^{i\bar{j}})$ be the transposed inverse of the matrix $(g_{i\bar{j}})$. For any two real $1$-forms $A$ and $B$ on $X$, locally given by
 \[
   A  = \sum_{i=1}^n \big(A_i d z_i + A_{\bar{i}} d \bar{z}_i \big) \quad \mbox{and} \quad
   B  = \sum_{i=1}^n \big( B_i d z_i + B_{\bar{i}} d \bar{z}_i \big),
 \]
we denote
 \[
    \langle A, B \rangle_{\omega} = \frac{1}{2}\sum_{i,j = 1}^n g^{i\bar{j}} \big( A_i B_{\bar{j}} + A_{\bar{j}} B_i \big).
 \]
 We may omit the subscript $\omega$ in $\langle \cdot, \cdot \rangle_{\omega}$ when it is understood from the context.
 In particular, we have
 \[
    \langle d h, d h \rangle =  \sum_{i,j=1}^n g^{i\bar{j}}\frac{\partial h}{\partial z_i}\frac{\partial h}{\partial \bar{z}_j} \equiv |\nabla h|^2, \qquad \textup{for all $h \in C^1(X)$}.
 \]
The Laplacian $\Delta$ associated with $\omega$ is given by
\[
   \Delta h = \frac{n \omega^{n-1} \wedge \ppr h}{\omega^n} = \sum_{i,j=1}^n g^{i\bar{j}} h_{i\bar{j}}, \qquad \textup{for all $h \in C^2(X)$}.
\]

We use the continuity method to solve \eqref{eq:sl1}. Fix an integer
$l \ge n+4$ and a real number $0 < \alpha <1$. We denote by
$C^{l,\alpha}(X)$ the usual H\"older space on $X$. Let
\[
   S(u) = \Delta u + \psi(|\nabla u|^2) + \langle B, du \rangle
   - \frac{\int_X (\Delta u + \psi(|\nabla u|^2) + \langle B, du \rangle) \, \omega^n }{\int_X \omega^n},
\]
for each $u \in C^{l,\alpha}(X)$. 
Consider the following family of equations,
\begin{eqnarray} \label{eq:slt}
 S(v_t) = t f, \qquad \textup{$0 \le t \le 1$}.
\end{eqnarray}
Let $I$ be the subset of $[0,1]$ consisting of $t$ for which the
equation \eqref{eq:slt} has a solution $v_t \in C^{l,\alpha}(X)$
satisfying
\begin{eqnarray} \label{eq:nmlt}
\int_X v_t \, \omega^n=0.
\end{eqnarray}
Obviously, the set $I$ is nonempty since $0 \in I$. The openness of $I$ will follow from our previous results~\cite[Section 3]{FuWangWu}. Indeed, let
\begin{equation}\lab{101}
   \mathcal{E}^{l,\alpha}_\omega = \left\{ h \in C^{l,\alpha}(X); \int_X h \omega^n = 0 \right\}.
\end{equation}
Notice that $S : \mathcal{E}^{l+2,\alpha}_\omega \to
\mathcal{E}^{l,\alpha}_\omega$. The linearization of $S$ is
\[
   L_{\omega} (h) = \left. \frac{d }{d t} S(v + th) \right|_{t = 0} = \Delta h + \langle \tilde{B}, d h\rangle - \frac{\int_X (\Delta h + \langle \tilde{B}, d h \rangle) \omega^n }{\int_X \omega^n},
\]
where
\[
  \tilde{B} = B + 2 \psi'(|\nabla v|^2) \, d v.
\]
It follows from the proof of Lemma 13 in \cite{FuWangWu} that
$L_{\omega}$ is a linear isomorphism from
$\mathcal{E}^{l+2,\alpha}(X)$ to $\mathcal{E}^{l,\alpha}(X)$. Thus,
by the implicit theorem we obtain the openness of $I$.

For the closedness of $I$ we need the \emph{a priori} estimate, which will be established in Section~\ref{se:sl}.

\section{A Prior estimates} \label{se:sl}
 Let $(X,\omega)$ be an $n$-dimensional hermitian
 manifold, $B$ a smooth $1$-form on $X$,  $f$ a smooth function on $X$, $c$ a constant, and $\psi \in C^{\infty}(\mathbb{R})$ satisfy \eqref{eq:defpsi}.
 Consider the following semi-linear equation:
\begin{equation} \label{eq:sl}
    S(v) \equiv \Delta v + \psi(|\nabla v|^2) + \langle B, dv\rangle - c = f  \qquad \textup{on
    $X$},
\end{equation}
where $v \in C^3(X)$ satisfies the normalization condition
\begin{eqnarray}\label{eq:nm1}
\int_X v \, \omega^n=0.
\end{eqnarray}
We shall first derive a uniform gradient estimate:
\begin{lemm} \label{le:C1}
  Let $v \in C^3(X)$ be a solution of \eqref{eq:sl}. We have
\[
    \sup_X | \nabla v | \le C,
\]
where $C > 0$ is a constant depending only on $B$, $f$,
$\omega$, $\psi(0)$, $\mu$ and $\nu$.
\end{lemm}
Throughout this section, we always denote by $C > 0$ a generic
constant depending only on $B$, $f$, $\omega$, $\psi(0)$, $\mu$, and $\nu$, unless otherwise
indicated.

\begin{proof}
Since $X$ is compact, we can assume that $|\nabla v|^2$ attains its
maximum at some point $x_0 \in X$. Consider the following linear
elliptic operator
   \[
      L(h) = \Delta h + 2 \psi'(|\nabla v|^2) \langle dh, dv \rangle_{\omega} = \Delta h + \psi'(|\nabla v|^2) g^{i\bar{j}}(h_i v_{\bar{j}} + h_{\bar{j}} v_i),
   \]
   Here the summation convention is used, and we denote
   \[
      h_i = \frac{\partial h}{\partial z^i}, \quad g^{i\bar{j}}_{,k} = \frac{\partial g^{i\bar{j}}}{\partial z^k}, \quad \cdots.
   \]
    We compute that
   \begin{align*}
      L(|\nabla v|^2)
      & = \Delta (|\nabla v|^2) + \psi' g^{i\bar{j}} \big[(|\nabla v|^2)_i v_{\bar{j}} + v_i (|\nabla v|^2)_{\bar{j}} \big] \\
      & = g^{i\bar{j}}g^{p\bar{q}}(v_{pi}v_{\bar{q}\bar{j}}+ v_{p\bar{j}} v_{i\bar{q}})
       +g^{p\bar{q}} \big[(\Delta v)_p v_{\bar{q}} + v_p (\Delta v)_{\bar{q}}\big]
      +      g^{i\bar{j}} g^{p\bar{q}}_{,i\bar{j}} v_p v_{\bar{q}}
      \\
      & \quad + g^{i\bar{j}} g^{p\bar{q}}_{,i}(v_{p\bar{j}}v_{\bar{q}} + v_p v_{\bar{q}\bar{j}})
              + g^{i\bar{j}} g^{p\bar{q}}_{,\bar{j}}(v_{pi}v_{\bar{q}} + v_p v_{i\bar{q}}) \\
      & \quad -g^{p\bar q}(g^{i\bar j}_{,p}v_{i\bar j}v_{\bar q}+g^{i\bar j}_{,\bar q}v_pv_{i\bar j})
      + \psi' g^{i\bar{j}} \big[(|\nabla v|^2)_i v_{\bar{j}} + v_i (|\nabla v|^2)_{\bar{j}}
      \big].
   \end{align*}
   Using   equation \eqref{eq:sl} to the second term  on the far
   right of above equalities and then using the Schwarz inequality, we find
   $$L(|\nabla v|^2)\ge \frac{1}{2} g^{i\bar{j}}g^{p\bar{q}}(v_{pi}v_{\bar{q}\bar{j}}
   + v_{p\bar{j}} v_{i\bar{q}}) - C |\nabla v|^2 - C.$$

   To see more clearly, let us take a normal coordinate system around $x_0$ such that
   \[
      g_{i\bar{j}}(x_0) = \delta_{ij}, \qquad \textup{for all $i,j= 1, \ldots, n$}.
   \]
   It follows that
   \begin{align*}
      L(|\nabla v|^2) & \ge \frac{1}{2} \sum_{i,p=1}^n |v_{p\bar{i}}|^2 - C |\nabla v|^2 - C \\
        & \ge \frac{1}{2} \sum_{i=1}^n |v_{i\bar{i}}|^2 - C |\nabla v|^2 - C \\
        & \ge \frac{1}{2n} |\Delta v|^2 - C |\nabla v|^2 - C \qquad \textup{\Big(by Cauchy's inequality\Big)}\\
        & \ge \frac{1}{2n} \big|\psi(|\nabla v|^2) + \langle B, dv \rangle - f- c\big|^2 - C |\nabla v|^2 - C \qquad \textup{\Big(by \eqref{eq:sl}\Big)}\\
        & \ge \frac{1}{4n} \big|\psi(|\nabla v|^2)\big|^2 - C |\nabla v|^2 - C(1 + |c|^2).
   \end{align*}
   We can assume, without loss of generality, that $|\nabla v|^2(x_0)$ is sufficiently large so that
   \[
       \psi (|\nabla v|^2) \ge \frac{\nu}{2} |\nabla v|^{2\mu} \qquad \textup{at $x_0$},
   \]
   where $\mu > 1/2$ and $\nu > 0$ are constants, by \eqref{eq:defpsi}.
   Now notice that
   \[
       L(|\nabla v|^2) \le 0 \quad \textup{at $x_0$},
   \]
   because of
   \[
      \Delta (|\nabla v|^2)(x_0) \le 0, \qquad \textup{and $\quad \nabla (|\nabla v|^2)(x_0) = 0$}.
   \]
   Hence, we obtain that
   \[
       \sup_X |\nabla v|^2 = |\nabla v|^2(x_0) \le C (1 + |c|^2).
   \]
   It remains to bound the constant $c$ in terms of $f$ and $\psi(0)$: Apply the usual maximum principle to \eqref{eq:sl} to obtain that
   \begin{equation} \label{eq:bddc}
       \psi(0) - \sup_X f \le c \le - \inf_X f + \psi(0).
   \end{equation}
   This finishes the proof.
\end{proof}

 Next, we establish the $C^0$ estimate: Noticing
\eqref{eq:nm1}, there must exist some point $y_0 \in X$ such that
$v(y_0)=0$. Then, for any point $y \in X$, we take a geodesic curve
$\gamma$ connecting $y_0$ to $y$. We have by  Lemma~\ref{le:C1}
that,
\begin{eqnarray}
|v(y)|=|v(y)-v(y_0)| = \left| \int_0^1 \frac{d (v \circ \gamma)}{dt}
dt \right| \le \int_0^1 (|\nabla v|\circ \gamma) \, dt <C.\nonumber
\end{eqnarray}
This settles the $C^0$ estimate of $v$.

We rewrite  equation \eqref{eq:sl} as
\begin{eqnarray}
\triangle v= - \psi(|\nabla v|^2) - \langle B, dv \rangle + f +
c.\nonumber
\end{eqnarray}
By $W^{2,p}$ theory of elliptic equations, we have for any $p > 1$,
\begin{align*}
\|v\|_{W^{2,p}}
& \leqq C(\|v\|_{L^p}+\| f+ c - \psi(|\nabla v|^2) - \langle B, dv \rangle \|_{L^p}) \\
& \le C_1,
\end{align*}
where in the last inequality we have
used the $C^0$ and $C^1$ estimates of $v$, and \eqref{eq:bddc}.
Here and below, we denote by $C_1$ a generic constant depending on $B$, $f$, $\omega$, $\mu$, $\nu$, and also $p$, and $\max \{|\psi(t)|; 0 \le t \le \max |\nabla v|^2 \le C \}$.

Fix a sufficiently large $p$ such that $\alpha \equiv 2n/p < 1$. It
follows from the Sobolev embedding theorem that
\begin{eqnarray}
\|v\|_{C^{1,\alpha}}\leq C_1. \nonumber
\end{eqnarray}
This allows us to apply Schauder's theory to obtain that
\[
   \|v\|_{C^{2,\alpha}} \le C_1.
\]
Thus, by the bootstrap argument, we have
\begin{eqnarray}
\|v\|_{C^{l,\alpha}}\leq C_1, \qquad \textup{for any $k \ge 1$}.
\end{eqnarray}
This implies that the set $I$ defined in Section~\ref{se:NP} is closed. As a consequence, we have shown the existence part in Theorem~\ref{th:sl1}.

\section{Uniqueness and Corollary} \label{se:UC}
Let us prove the uniqueness in Theorem~\ref{th:sl1}. Suppose that there exist $c$, $v$ and $\tilde{c}$, $\tilde{v}$ such that
\begin{align*}
  \Delta v + \psi(|\nabla v|^2) + \langle B, d v\rangle & = f + c, \\
  \Delta \tilde{v} + \psi(|\nabla \tilde{v} |^2) + \langle B, d \tilde{v} \rangle & = f + \tilde{c}.
\end{align*}
Then,
\begin{align}
   c & = \frac{\int_X (\Delta v + \psi(|\nabla v|^2) + \langle B, dv \rangle) \omega^n}{\int_X \omega^n}, \label{eq:ctc1}\\
   \tilde{c} & = \frac{\int_X (\Delta \tilde{v} + \psi(|\nabla \tilde{v}|^2) + \langle B, d\tilde{v} \rangle)\omega^n}{\int_X \omega^n}. \label{eq:ctc2}
\end{align}
Recall that we denote
\[
    S(u) = \Delta u + \psi(|\nabla u|^2) + \langle B, d u \rangle - \frac{\int_X (\Delta u + \psi(|\nabla u|^2) + \langle B, d u \rangle) \omega^n}{\int_X \omega^n}, 
\]
for all $u \in C^2(X)$.
It follows that
\begin{align}
   0 & = S(v) - S(\tilde{v}) = \int_0^1 \left[\frac{d}{d t} S\big(t v + (1 - t) \tilde{v}\big)\right] dt \notag \\
   & = \Delta w + \langle \tilde{B}, d w\rangle - c_w. \label{eq:linS}
\end{align}
Here $w = v - \tilde{v}$, 
\[
   \tilde{B} = B + 2 \int_0^1 \psi'(|t \nabla v + (1 - t) \nabla \tilde{v}|^2) \, \big[ t d v + (1 - t) d \tilde{v}\big] dt,
\]
and $c_w$ is a constant given by
\[
   c_w = \frac{\int_X ( \Delta w + \langle \tilde{B}, d w \rangle ) \omega^n}{\int_X \omega^n}.
\]
Applying the maximum principle to \eqref{eq:linS} yields
\[
    c_w = 0.
\]
Then, by the strong maximum principle we conclude that $w$ is equal to a constant. This shows that the solution of \eqref{eq:sl1} is unique up to a constant. By \eqref{eq:ctc1} and \eqref{eq:ctc2} we have $c = \tilde{c}$. This completes the proof of Theorem~\ref{th:sl1}.

Let us now prove Corollary \ref{co:sl2}. We define a smooth real
1-form on $X$
\begin{equation}\lab{102}
B_1=\frac {\sqrt{-1}}{2}\frac{nk}{n-1}\frac
{1}{n!}\ast\bigl(\partial(\omega^{n-1})-\bar\partial(\omega^{n-1})\bigr)
\end{equation}
and a smooth function
\begin{equation} \label{eq:defphi}
   \varphi= \frac{n \ppr (\omega^k)\wedge \omega^{n-k-1}}{\omega^n}.
\end{equation}
 Then \eqref{eq:sl2} is equivalent to
\[
   \Delta v + |\nabla v|^2 + \langle B_1, d v\rangle + \varphi = n\gamma_k.
\]
Letting
\[
   \psi (t) = t \quad \textup{and \quad} f = \frac{\int_X \varphi \omega^n}{\int_X \omega^n} - \varphi,
\]
Corollary~\ref{co:sl2} then follows readily from Theorem~\ref{th:sl1}.


   For each $1 \le k \le n-1$, the constant $\gamma_k$ is given by
  \begin{align}
    \gamma_k
    & = \frac{\int_X e^{-v} \ppr (e^v \omega^k )\wedge \omega^{n-k-1}}{\int_X \omega^n} \label{eq:defgam} \\
    & =  \frac{\int_X (\Delta v + |\nabla v|^2 + \langle B_1, dv \rangle + \varphi)\omega^n}{n \int_X \omega^n} \label{eq:gam1}.
  \end{align}
On the other hand, directly integrating \eqref{eq:sl2} over $X$ yields that
\begin{equation} \label{eq:gam2}
   \gamma_k = \frac{\int_X \ppr (e^v \omega^k)\wedge \omega^{n-k-1}}{\int_X e^v \omega^n}.
\end{equation}
This together with \eqref{eq:defgam} imposes some constraint on the constant $\gamma_k$. For instance, when $k = n - 1$, by \eqref{eq:gam2} we know that
\[
   \gamma_{n-1} = 0.
\]
Thus, in this case Corollary~\ref{co:sl2} recovers the classic result of Gauduchon~\cite{Gau}. When $\omega$ is K\"ahler, by \eqref{eq:gam2} again we have
\[
   \gamma_k = 0 \qquad \textup{for all $1 \le k \le n-1$}.
\]
Then, it follows from \eqref{eq:gam1} that
\[
   \int_X |\nabla v|^2 \omega^n = 0.
\]
This tells us that the solution $v$ of \eqref{eq:sl2} has to be a constant.
\vspace{1mm}

\section{Generalized Gauduchon metrics and $\gamma_k$} \label{se:GG}

Let $X$ be an $n$-dimensional complex manifold. We recall by Definition~\ref{de:kGau} that
a hermitian metric $\omega$ on $X$ is called $k$-th Gauduchon metric if
\begin{equation*}
\p \bar{\p}(\omega^k)\wedge \omega^{n-k-1}=0 \qquad \textup{on $X$}.
\end{equation*}
Then, the $(n-1)$--th Gauduchon metric is the Gauduchon
metric in the usual sense. By Corollary~\ref{co:sl2}, each
hermitian metric $\omega$ on $X$ can be associated with a unique constant
$\gamma_k(\omega)$, which is invariant under biholomorphisms.
The induced function $\gamma_k = \gamma_k(\omega)$ can be used to characterize the $k$-th Gauduchon metric.
\begin{prop} \label{pr:kGau}
  The hermitian manifold $X$ admits a $k$-th Gauduchon metric
  if and only if that there exists a hermitian metric $\omega$ on $X$ such that
  \begin{equation} \label{cond}
     \gamma_k (\omega) = 0.
  \end{equation}
\end{prop}
\begin{proof}
If there is some hermitian metric $\omega$ satisfying \eqref{cond}, then
Corollary~\ref{co:sl2} implies that the conformal metric $e^{v/k}\omega$ is a
$k$-th Gauduchon metric on $X$. Conversely, if $\omega$ is a $k$-th Gauduchon metric,
then the uniqueness of Corollary~\ref{co:sl2} implies that $\gamma_k(\omega)=0$
and that $v$ is a constant.
\end{proof}
Let $\mathfrak{M}$ be the set of all hermitian metrics on $X$. We shall prove that
$\gamma_k$ is a smooth function on $\mathfrak{M}$. Here
$\mathfrak{M}$ is viewed as an open subset in $C^{l+2,\alpha}(\Lambda^{1,1}_{\mathbb{R}}(X))$, for a
nonnegative integer $l$ and a real number $0 < \alpha < 1$. 
We denote by
$C^{l,\alpha}(\Lambda^{m,m}_{\mathbb{R}}(X))$ the H\"older space of
real $(m,m)$--forms on $X$, in which $l$ and $m$ are nonnegative integers, and
$0 < \alpha<1$ is a real number. In particular, $C^{l,\alpha}(\Lambda^{0,0}_{\mathbb{R}}(X)) = C^{l,\alpha}(X)$.


\begin{prop} \label{pr:moduli}
The function $\gamma_k = \gamma_k(\omega)$ is a smooth function on
$\mathfrak{M}$, where $\mathfrak{M}$ is viewed as an open subset in $C^{l+2,\alpha}(\Lambda^{1,1}_{\mathbb{R}}(X))$.
\end{prop}
\begin{proof}
It follows from Corollary~\ref{co:sl2} that, for each $\omega \in \mathfrak{M}$,
there exists a unique constant $\gamma_k$ and a function $v$ such that
\begin{equation} \label{eq:sl3}
   e^{-v} \ppr (e^v \omega^k) \wedge \omega^{n-k-1} - \gamma_k \omega^n = 0.
\end{equation}
Then,
\[
   \gamma_k = \frac{\int_X e^{-v} \ppr (e^v \omega^k)\wedge \omega^{n-k-1}}{\int_X \omega^n}
\]
depends smoothly on $v$ and $\omega$. Thus, to show the result, it suffices
to show that the solution $v$ depends smoothly on $\omega$. We shall use the implicit function theorem.

For each $\omega \in \mathfrak{M}$, the space
$\mathcal{E}^{l,\alpha}_{\omega}$ is defined  by (\ref{101}).
Fix $\omega_0 \in \mathfrak{M}$, for which we abbreviate
$\mathcal{E}^{l,\alpha}_0= \mathcal{E}^{l,\alpha}_{\omega_0}$. We
have two obvious linear isomorphisms
from $\mathcal{E}^{l,\alpha}_{\omega}$ to
$\mathcal{E}^{l,\alpha}_0$, given respectively by
\begin{equation} \label{eq:shift}
   h \longmapsto h - \frac{\int_X h \omega_0^n}{\int_X \omega_0^n},
   \qquad \textup{for all $h \in \mathcal{E}^{l,\alpha}_{\omega}$},
\end{equation}
and
\begin{equation} \label{eq:multi}
   h \longmapsto h \cdot \frac{\omega^n}{\omega_0^n} \qquad \textup{for all $h \in \mathcal{E}^{l,\alpha}_{\omega}$}.
\end{equation}

Define a map $F: \mathfrak{M}\times \mathcal{E}^{l+2,\alpha}_0 \to
\mathcal{E}^{l,\alpha}_0 $ by 
\begin{align*}
 F(\omega, v) & = \frac{n e^{-v}\ppr
(e^v\omega^{k}) \wedge \omega^{n-k-1}}{\omega_0^n} \\
  & \quad - \frac{n\int_X e^{-v} \ppr (e^v \omega^k) \wedge \omega^{n-k-1}}{\int_X \omega^n} \cdot \frac{\omega^n}{\omega^n_0}.
\end{align*}
Obviously, $F$ is a smooth map. Note that any $(\omega,v)\in
\mathfrak{M}\times \mathcal{E}_0^{l+2,\alpha}$ satisfies
\eqref{eq:sl3} if and only if
\[
   F(\omega, v) = 0.
\]
The Fr\'echet derivative of $F$ with respect to
the variable $v$ is
$$D_v F(\omega,v) (h) =L_{\omega}(h)\frac{\omega^n}{\omega_0^n}.$$
Here
\[
   L_{\omega} (h)  = \Delta h + \langle B_1+2dv, dh \rangle_{\omega}
   - \frac{\int_X (\Delta h + \langle B_1+2dv, d h \rangle_{\omega}) \omega^n}{\int_X \omega^n},
\]
in which the Laplacian $\Delta$ is with respect to $\omega$,
and $B_1$ is the smooth real $1$-form given by \eqref{102}.
By the proof of Lemma 13 in \cite{FuWangWu} and the isomorphism
\eqref{eq:shift}, the operator $L_{\omega} :
\mathcal{E}^{l+2,\alpha}_{0}\to \mathcal{E}^{l,\alpha}_{\omega}$ is
a linear isomorphism. This combining isomorphism \eqref{eq:multi}
imply that $D_v F(\omega,v): \mathcal{E}^{l+2,\alpha}_0 \to
\mathcal{E}^{l,\alpha}_0$ is a linear isomorphism. The result then
follows by the Implicit Function Theorem.
\end{proof}
A direct corollary of Proposition~\ref{pr:moduli} is as below.
\begin{coro}\label{cor0}
For $1 \le k \le n-2$, if there exists two hermitian metric $\omega_1,\omega_2$ on $X$
such that
$$\gamma_k(\omega_1)>0\ \ \text{and} \ \ \gamma_k(\omega_2)<0,$$
then there exists a metric $\omega$ on $X$ satisfying $\gamma_k(\omega)=0$, i.e., $\omega$ is a $k$-th Gauduchon metric.
\end{coro}

\begin{proof}
Let
$$\omega_t=t\omega_1+(1-t)\omega_2, \qquad \textup{for all $0 \le t \le 1$}.$$
Then $\omega_t$ is a hermitian metric for each $t$. The result follows immediately by applying
the Mean Value Theorem to the function $\phi(t) = \gamma_k (\omega_t)$.
\end{proof}

\begin{prop} \label{pr:cfinv}
  For any function $\rho \in C^2(M)$, we have
  \begin{equation} \label{eq:compv}
     e^{-\max_X \rho} \gamma_k(\omega) \le \gamma_k (e^{\rho}\omega) \le e^{-\min_X \rho} \gamma_k(\omega).
  \end{equation}
  In particular, the sign of the function $\gamma_k$ is a conformal invariant for hermitian metrics.
\end{prop}
\begin{proof}
  Let $\tilde{\omega} = e^{\rho}\omega$. Then, there exists a function $\tilde{v}$ and a number $\tilde{\gamma}_k = \gamma_k(\tilde{\omega})$ satisfying
  \[
     \ppr (e^{\tilde{v}} \tilde{\omega}^k) \wedge \tilde{\omega}^{n-k-1} = \tilde{\gamma}_k e^{\tilde{v}} \tilde{\omega}^n,
  \]
  that is,
  \begin{equation}\label{eq:torho}
     \ppr (e^{\tilde{v} + k \rho} \omega^k) \wedge \omega^{n-k-1} = \tilde{\gamma}_k e^{\tilde{v} + k \rho} e^{\rho} \omega^n.
  \end{equation}
  We can rewrite \eqref{eq:torho} as
  \begin{equation} \label{eq:tlv}
     \Delta (\tilde{v} + k \rho) + |\nabla (\tilde{v} + k \rho)|^2 + \langle B_1, d (\tilde{v} + k \rho) \rangle + \varphi =n e^{\rho} \tilde{\gamma},
  \end{equation}
  where the operators $\Delta$ and $\nabla$ are with respect to $\omega$, and $B_1$ and $\varphi$ are given
  by (\ref{102}) and \eqref{eq:defphi}, respectively. Subtracting \eqref{eq:tlv} by
  \[
     \Delta v + |\nabla v|^2 + \langle B_1, d v\rangle + \varphi = n\gamma_k(\omega)
  \]
  and then applying the maximum principle yields \eqref{eq:compv}.
\end{proof}

\begin{prop} \label{pr:0term}
  For a hermitian metric $\omega$, the number $\gamma_k(\omega) > 0$ $(= 0$, or $< 0)$ if and only if there exists a metric $\tilde{\omega}$ in the conformal class of $\omega$ such that
  \begin{equation} \label{eq:0term}
     \ppr \tilde{\omega}^k \wedge \tilde{\omega}^{n-k-1} > 0 \, \textup{$( = 0$, or $<0$)} \quad \textup{on $X$}.
  \end{equation}
\end{prop}
\begin{proof}
  Suppose that $\gamma_k(\omega)>0$ $(=0$, or $< 0$). Let $\tilde{\omega} = e^{v/k} \omega$,
  where $v$ is the smooth function associated with $\omega$ so that \eqref{eq:sl2} holds. Then,
  \[
     \ppr \tilde{\omega}^k \wedge \tilde{\omega}^{n-k-1} = \gamma_k(\omega) \omega^n e^{(n-k)v} > 0 \, \textup{$(=0$, or $<0$)}.
  \]
  Conversely, if there is a metric $\tilde{\omega}$ in the conformal class of $\omega$ such that \eqref{eq:0term} holds, then we claim that $\gamma_k(\tilde{\omega}) > 0$ $(=0$, or $<0$). Indeed, by Corollary~\ref{co:sl2} there exists a smooth function $\tilde{v}$ such that
  \[
     \ppr (e^{\tilde{v}}\tilde{\omega}^k) \wedge \tilde{\omega}^{n-k-1} = \gamma_k(\tilde{\omega}) e^{\tilde{v}}\tilde{\omega}.
  \]
  This is equivalent to the following equation
  \begin{equation} \label{eq:tform}
     \Delta \tilde{v} + |\nabla \tilde{v}|^2 + \langle \tilde B_1, d\tilde{v} \rangle + \tilde\varphi = n\gamma_k(\tilde{\omega}),
  \end{equation}
  where the operators $\Delta$ and $\nabla$ are with respect to $\tilde{\omega}$, and $\tilde B_1$
  and $\tilde\varphi$ 
  are given by (\ref{102}) and \eqref{eq:defphi}, respectively,
  with $\tilde\omega$  replacing $\omega$.
  By \eqref{eq:0term} we have $\tilde{\varphi} >0$ $(=0$, or $<0)$. The claim then follows immediately by applying the maximum principle to \eqref{eq:tform}. By Proposition~\ref{pr:cfinv}, we finish the proof.
\end{proof}

Moreover, for the case of $\gamma_k >0$, we have the following criteria on the integration,
which is often easier to verify.
\begin{lemm}\label{10}
Suppose that $n$, the complex dimension of $X$, is an odd number. Let $k = (n-1)/2$.
Then, there is some metric $\omega$ satisfying $\gamma_k(\omega)>0$ if
and only if there is some semi-metric $\mathring{\omega}$ \textup{(}i.e., semi-positive
real $(1,1)$-form on $X$\textup{)} satisfying
$$\fr \int_X  \p \bar{\p} \mathring{\omega}^k \wedge\mathring{\omega}^{n-k-1} >0$$
\end{lemm}
\begin{proof}
By Proposition~\ref{pr:0term}, the necessary part is obvious. For the sufficient
part, let $\hat{\omega}$ be any hermitian metric. Let
$$\omega_{t}=\mathring{\omega}+t\hat{\omega}$$  for $t\in(0,1)$. Then we have
\begin{equation} \label{eq:gamint1}
 \begin{split}
& \int_X e^{-v}\ppr (e^v\omega_t^k)\wedge\omega_t^{n-k-1} \\
&=\frac{\sqrt{-1}}{2}\int_X \big(\p\bar{\p}\omega_t^k \wedge\omega_t^{n-k-1} +  \p
v\wedge\bar{\p}v\wedge\omega_t^{n-1}\big) \\
& \quad +\frac{\sqrt{-1}}{2}\int_X \left[ \p\bar{\p}v\wedge
\omega_t^{n-1} + \frac{k}{n-1} \left(\p\omega_t^{n-1}\wedge
\bar{\p}v + \p v\wedge \bar{\p}\omega_t^{n-1}\right) \right]\\
&=\frac{\sqrt{-1}}{2}\int_X \big(\p\bar{\p}\omega_t^k
\wedge\omega_t^{n-k-1} +  \p
v\wedge\bar{\p}v\wedge\omega_t^{n-1}\big) \\
& \quad +\frac{\sqrt{-1}}{2}\Big(1-\frac{2k}{n-1}\Big)\int_X  v\p\bar{\p}
\omega_t^{n-1}.
  \end{split}
\end{equation}
Since $k = (n-1)/2$, the second integral on the right of
\eqref{eq:gamint1} vanishes. It follows that
\begin{align*}
& \int_X e^{-v}\ppr (e^v\omega_t^k)\wedge\omega_t^{n-k-1}
\ge \frac{\sqrt{-1}}{2}\int_X\p\bar{\p}\omega_t^k \wedge\omega_t^{k} \\
& = \frac{\sqrt{-1}}{2}\int_X\p\bar{\p}\mathring{\omega}^k \wedge \mathring{\omega}^k +t\frac{\sqrt{-1}}{2}\int_X(\p\bar{\p}\mathring{\omega}^k \wedge\Psi_t+
\p\bar{\p}\Psi_t \wedge\mathring{\omega}^k) \\
& \quad +t^2 \frac{\sqrt{-1}}{2}\int_X\p\bar{\p}\Psi_t\wedge\Psi_t > 0, \qquad \textup{for sufficiently small $t$},
\end{align*}
where $\Psi_t = \hat{\omega} \wedge (\mathring{\omega}^{k-1} + \mathring{\omega}^{k-2}\wedge \omega_t + \cdots + \mathring{\omega} \wedge \omega_t^{k-2} + \omega_t^{k-1})$. This implies that $\gamma_k(\omega_t) >0$ for the sufficiently small $t$.
\end{proof}

A similar argument works for the (classic) Gauduchon metrics, for any
dimension $n$, and for all $1 \le k \le n-2$.
\begin{lemm}
Let $X$ be an $n$-dimensional hermitian manifold, $k$ an integer such that $1 \le k \le n-2$.
Then, a hermitian metric $\omega$ on $X$ satisfies
$\gamma_k(\omega)>0$ if the Gauduchon metric $\tilde{\omega}$
in the conformal class of $\omega$ satisfies
\begin{equation} \label{eq:intc2}
     \fr \int_X \p \bar{\p} \tilde{\omega}^k \wedge \tilde{\omega}^{n-k-1} > 0.
  \end{equation}
\end{lemm}
\begin{proof}
 By Proposition~\ref{pr:cfinv}, we can assume that $\omega = \tilde{\omega}$, without loss of generality.
  By \eqref{eq:gamint1} with $\omega$ replacing $\omega_t$, and applying $\p \bar{\p} \omega^{n-1} = 0$ yields
  \begin{align*}
   \int_X e^{-v}\ppr (e^v\omega^k)\wedge\omega^{n-1-k}
 \ge \frac{\sqrt{-1}}{2}\int_X\p\bar{\p}\omega^k  \wedge\omega^{n-1-k} > 0.
  \end{align*}
\end{proof}
\begin{coro}
  Let $(X,\omega)$ be an $n$-dimensional balanced manifold. Then, for each $1 \le k \le n-2$, we have $\gamma_k(\omega)> 0$ if
  \begin{equation*}
     \fr \int_X \p \bar{\p} \omega^k \wedge \omega^{n-1-k} > 0.
  \end{equation*}
\end{coro}

\section{Constructions on hermitian three--manifolds} \label{se:posgam}
We shall apply previous results to  construct a hermitian metric with $\gamma_1 >0$ on a complex three dimensional manifold. Theorem~\ref{th:posgam} will follow from Proposition~\ref{pr:0term}, together with the following theorem.

\begin{theo}\label{12}
There always exists a hermitian metric $\omega$ on a complex three dimensional manifold $X$ such that
$$(\sqrt{-1}/2) \p\bar{\p}\omega\wedge\omega >0.$$
\end{theo}
\begin{proof}
By Lemma~\ref{10} and Proposition~\ref{pr:0term}, it suffices to construct a semi-metric $\mathring{\omega}$ such that
\[
        \fr \int_X  \p \bar{\p} \mathring{\omega} \wedge \mathring{\omega} >0.
\]

Fix a point $q\in X$ and a  coordinate patch $U \ni q$. Let $(z_1,z_2,z_3)$
be coordinates on $U$ centered at $q$. Here $z_j=x_j+\sqrt{-1}y_j$
for $1\leq j\leq 3$. We can assume $N= B \times B \times R \subset
U$, where $B$ is the unit ball in $\mathbb{C}$, and
\[
   R = \{z_3\in \mathbb{C}\mid |x_3| \le 1, |y_3| \le 1 \}.
\]

Take a nonnegative cut-off function   $\eta \in C^{\infty}_0(B)$ and
 two nonnegative functions $f, g \in C^{\infty}_0([-1,1])$ to be
determined later. On
  $N$, define   $$\phi =\eta(z_1)\eta(z_2)f(x_3)f(y_3),\ \
\psi =\eta(z_1)\eta(z_2)g(x_3)g(y_3),$$ and then define
\begin{eqnarray}
\mathring{\omega}=\frac{\sqrt{-1}}{2}\big[\phi(z) dz_1\wedge
d\bar{z}_1+\psi(z) dz_2\wedge d\bar{z}_2\big].
\end{eqnarray}
 Obviously,
$\mathring{\omega}$ is semi-positive and with compact support in
$N$. So it can be viewed  as a semi-metric on $X$. Clearly,
\begin{eqnarray}
\ppfr\mathring{\omega}\wedge\mathring{\omega} = \Big(\phi
\frac{\p^2\psi}{\p z_3\p\bar{z}_3}+\psi \frac{\p^2\phi}{\p
z_3\p\bar{z}_3}\Big) dV,
\end{eqnarray}
where
\begin{equation} \label{eq:defdV}
   dV = \Big(\frac{\sqrt{-1}}{2}\Big)^3dz_1\wedge d\bar{z}_1\wedge
dz_2\wedge d\bar{z}_2\wedge dz_3\wedge d\bar{z}_3.
\end{equation}
Since $$\frac{\p}{\p z_3}=\frac{1}{2}\Big(\frac{\p}{\p
x_3}-\sqrt{-1}\frac{\p}{\p y_3}\Big),\ \ \frac{\p}{\p
\bar{z}_3}=\frac{1}{2}\Big(\frac{\p}{\p x_3}+\sqrt{-1}\frac{\p}{\p
y_3}\Big),
$$
we have
\begin{align*}
& \phi\frac{\p^2\psi}{\p z_3\p\bar{z}_3}+\psi \frac{\p^2\phi}{\p
z_3\p\bar{z}_3} \\
& = \frac{\phi}{4} \left(\frac{\p^2 \psi}{\p x_3 \p x_3} + \frac{\p^2 \psi}{\p y_3 \p y_3} \right)
+ \frac{\psi}{4} \left(\frac{\p^2 \phi}{\p x_3 \p x_3} + \frac{\p^2 \phi}{\p y_3 \p y_3}\right) \\
&=\frac{1}{4}\eta^2(z_1)\eta^2(z_2)f(y_3)g(y_3)\big[f(x_3) g''(x_3)+g(x_3)f''(x_3)\big]\\
& \quad
+\frac{1}{4}\eta^2(z_1)\eta^2(z_2)f(x_3)g(x_3)\big[f(y_3)g''(y_3)+g(y_3)f''(y_3)\big].
\end{align*}
We choose  $\eta$ so that
\[
   \int_B \eta^2(z) \fr d z \wedge d \bar{z} = 1.
\]
Then it follows that
\begin{align*}
\fr \int_X \p \bar{\p} \mathring{\omega} \wedge \mathring{\omega}
&= \frac{1}{2}\int_{-1}^1 f(t) g(t) dt \int^1_{-1} \big[f(t)g''(t) + f''(t) g(t)\big] dt \\
&= \int_{-1}^1 f(t) g(t) dt \int^1_{-1} \big[ - f'(t) g'(t) \big]
dt.
\end{align*}
The result follows immediately from the proposition below.
\end{proof}

\begin{prop}
 There exist nonnegative functions $f, g \in C^{\infty}_0([-1,1])$ such that
 \[
    - \int_{-1}^1 f'(t) g'(t) d t > 0.
 \]
\end{prop}
\begin{proof}
 For any two real numbers $a< b$, we denote
 \[
   \chi_{a,b} (t) =
   \begin{cases}
    \exp\left(\dfrac{1}{t - b} - \dfrac{1}{t - a}\right), & \textup{if $a < t < b$}, \\
    0, & \textup{otherwise}.
   \end{cases}
 \]
 Clearly, we have that $\chi_{a,b} \in C_0^{\infty}(\mathbb{R})$, that $\chi_{a,b}'(t) >0$ for $a < t < (a+b)/2$, that $\chi_{a,b}'(t) < 0$ for $(a+b)/2 < t < b$, and that $\chi'_{a,b}(t) = 0$ when $t = (a+b)/2$. Letting
 \[
    f(t) = \chi_{-1/3,1/3}(t), \qquad \textup{and \quad $g(t) = \chi_{0,2/3}(t)$}
 \]
 yields that $- f'(t) g'(t) > 0$ for $0 < t < 1/3$ and otherwise $f'(t)g'(t) = 0$. This in particular implies the result.
\end{proof}

Let us now consider some examples. We can directly construct a
hermitian metric $\omega$ with $\gamma_1(\omega) >0$ on $T^3$, the
$3$-dimensional complex torus.

\begin{prop} \label{pr:torus1}
On the complex $T^3$, there is a metric $\omega$ satisfying
$$\sqrt{-1}/2\p\bar{\p}\omega\wedge\omega >0.$$
\end{prop}
\begin{proof}
Let $(z_1, z_2, z_3)$ be the coordinates of $T^3$
induced from $\mathbb{C}^3$. Let
$$\omega=\frac{\sqrt{-1}}{2}\left[\xi(x_3)dz_1\wedge d\bar{z}_1+\eta(x_3)dz_2\wedge d\bar{z}_2
+dz_3\wedge d\bar{z}_3\right],$$
where $\xi$ and $\eta$ are two positive smooth functions on $T^3$ only depending on $x_3$,
which will be determined later.
Then
\begin{equation*}
\ppr\omega\wedge\omega  =\left(\eta\frac{\p^2 \xi}{\p z_3\p
\bar{z}_3}+\xi\frac{\p^2 \eta}{\p z_3\p
\bar{z}_3}\right)dV>0
\end{equation*}
if and only if
$$\eta\frac{\p^2 \xi}{\p z_3\p
\bar{z}_3}+\xi\frac{\p^2 \eta}{\p z_3\p \bar{z}_3} =\frac 1 4
\eta\frac{\p^2 \xi}{\p x_3^2}+\frac 1 4\xi\frac{\p^2 \eta}{\p
x_3^2}>0.$$ Here $dV$ is defined by \eqref{eq:defdV}.
So we need to look for  two smooth, positive,
$2\pi$-periodic functions $\eta$ and $\xi$ such that
$$\frac{\eta''(t)}{\eta(t)}+\frac{\xi''(t)}{\xi(t)}>0.$$

We define
\begin{equation} \label{eq:defxi}
  \xi(t)=1+ \kappa \sin t, \qquad \textup{for some $0 < \kappa <1$}.
\end{equation}
We observe that
\begin{eqnarray}
\int^{2\pi}_0\frac{\xi''}{\xi} dt =-\int^{2\pi}_0\frac{\kappa \sin t}{1+\kappa  \sin
t} dt =-2\pi+\int^{2\pi}_0\frac{dt}{1+ \kappa \sin t} \nonumber.
\end{eqnarray}
By Proposition 8 in \cite{FuWangWu}, the
value of above integral tends to $+\infty$ monotonically, as $\kappa \to 1^-$.  Hence, for a
constant $C > 0$, there is a unique real number $\kappa$, such that the function $\xi$ given
by \eqref{eq:defxi} satisfies
$$\int^{2\pi}_0\frac{\xi''}{\xi} dt =\int_0^{2\pi}C dt.$$ It implies that equation
$$\zeta''+\frac{\xi''}{\xi}= C$$ has a  smooth $2\pi$-periodic solution $\zeta$ on $\mathbb{R}$. Let $\eta=e^\zeta$. Thus,
\begin{eqnarray}
\frac{\eta''(t)}{\eta(t)}+\frac{\xi''(t)}{\xi(t)}&=&(\zeta')^2+\zeta''+\frac{\xi''}{\xi}\ \ \geq \ \
C\ \  > \ \ 0.\nonumber
\end{eqnarray}
\end{proof}
As another example, we show that the natural balanced metric on the
Iwasawa manifold has positive $\gamma_1$. Recall (for example,
\cite[p. 444]{GrHa} and \cite[p. 115]{MK}) that the Iwasawa manifold
is defined to be the quotient space $G/\Gamma$, where
\[
   G = \left\{ \left[\begin{matrix}
                       1 & z_1 & z_3 \\
                       0 & 1 & z_2 \\
                       0 & 0 & 1
                \end{matrix} \right]; z_1, z_2, z_3 \in \mathbb{C}\right\},
\]
$\Gamma$ is the discrete subgroup of $G$ consisting of matrices where $z_1$, $z_2$, $z_3$ are Gaussian integers, i.e., $z_i \in \{a + b \sqrt{-1} \mid a, b \in \mathbb{Z}\}$ for $1 \le i \le 3$,
and $\Gamma$ acts on $G$ by left multiplications. Clearly, the global holomorphic 1-forms
\[
   \varphi_1 = dz_1, \qquad \varphi_2 = d z_2, \qquad \varphi_3 = dz_3 - z_1 dz_2 
\]
on $G$ are invariant under the action of $\Gamma$, hence descend down to $G/\Gamma$. Observe that $G/\Gamma$ does not admit any K\"ahler metric, because $d \varphi_3 = \varphi_2 \wedge \varphi_1 \ne 0$. Let
\[
   \omega = (\sqrt{-1}/2) (\varphi_1 \wedge \bar{\varphi}_1 + \varphi_2 \wedge \bar{\varphi}_2 + \varphi_3 \wedge \bar{\varphi}_3).
\]
Then, $\omega$ is a balanced hermitian metric on $G/\Gamma$, for $d \omega^2 = 0$. Furthermore, we have
\[
 (\sqrt{-1}/2) \p \bar{\p} \omega \wedge \omega 
 = (\sqrt{-1}/2)^3 \varphi_1 \wedge \bar{\varphi}_1 \wedge \varphi_2 \wedge \bar{\varphi}_2 \wedge \varphi_3 \wedge \bar{\varphi}_3 > 0
\]
on $G/\Gamma$; hence, by Proposition~\ref{pr:0term}, we conclude that $\gamma_1(\omega) >0$.

\section{The first Gauduchon metric on Calabi's manifolds}
In this section, we shall establish the existence of the $1$-st
Gauduchon metric on the non-K\"ahler manifold introduced by
Calabi~\cite{Ca}. In view of Theorem~\ref{th:posgam} and
Corollary~\ref{cor0}, we need to find a hermitian metric with
negative $\gamma_1$ value.

We first recall Calabi's construction of non-K\"ahler complex three
dimensional manifolds.  Let $\mathbb O\cong\mathbb R^8$ denotes the
Cayley numbers. We fix a
basis $\{I_1,\cdots, I_7\}$ such that\\
(1) $I_i\cdot I_j=\delta_{ij}$ with respect to the inner product.\\
(2) The table of the multiplication of the cross product $I_j\times
I_k$ is the following
\begin{equation}\lab{301}
\begin{array}{c|ccccccc}
\times & I_1& I_2&I_3&I_4&I_5&I_6&I_7\\
\hline I_1&0&I_3&-I_2&I_5&-I_4&I_7&-I_6\\
I_2&-I_3&0&I_1&I_6&-I_7&-I_4&I_5\\
I_3&I_2&-I_1&0&-I_7&-I_6&I_5&I_4\\
I_4&-I_5&-I_6&I_7&0&I_1&I_2&-I_3\\
I_5&I_4&I_7&I_6&-I_1&0&-I_3&-I_2\\
I_6&-I_7&I_4&-I_5&-I_2&I_3&0&I_1\\
I_7&I_6&-I_5&-I_4&I_3&I_2&-I_1&0
\end{array}
\end{equation}
Via this basis, we have the isomorphism $\mathbb R^7\cong
\text{Im}(\mathbb O)$.

Calabi considered a smooth oriented hypersurface $X^6\hookrightarrow
\mathbb R^7$. Fix a unit normal vector field $N$ of $X$. There is a
natural almost complex structure $J:TX\to TX$ induced by Cayley
multiplication as follows. For any $x\in X$ and any $V\in T_xX$,
define $J:T_xX\to T_xX$ as
\begin{equation*}
J(V)=N\times V.
\end{equation*}
Calabi  proved that $J$ is integrable if and only if $J$
anticommutes with the second fundamental form of $X$.

Calabi constructed  compact complex manifolds as follows. Let
$\Sigma$ be a compact Riemann surface which admits 3 holomorphic
differentials $\phi_1$, $\phi_2$, $\phi_3$ with the following
properties:
\begin{enumerate}
\item linear independent;
\item  $\phi_1^2+\phi_2^2+\phi_3^2=0$;
\item $\phi_1\wedge\bar\phi_1+\phi_2\wedge\bar\phi_2+\phi_3\wedge\bar
\phi_3>0$.
\end{enumerate}
Lifting $\phi_1$, $\phi_2$, $\phi_3$ to the universal covering
$\tilde{\Sigma}\to \Sigma$ and setting
\begin{equation*}
x^j(p)=\text{Re}\int_{p'}^p\phi_j,\ \ j=1,2,3
\end{equation*}
for a fixed point $p'\in\Sigma $, we obtain a conformal minimal
immersion
\begin{equation*}
\psi=(x^1,x^2,x^3):\tilde{\Sigma}\to \mathbb R^3.
\end{equation*}
This mapping is regular, since the differentials $\phi_j$ satisfy
(3); by the weierstrass representation, property (2) is equivalent
to the statement that $\psi$ is minimal; finally, because of
property (1), it follows that $\tilde{\Sigma}$ is not mapped into a
plane.

Calabi then  considered the hypersurface of the type
\begin{equation*}
(\psi,id):\tilde{\Sigma}\times \mathbb R^4\to\mathbb R^3\times
\mathbb R^4=\text{Im}(\mathbb O),
\end{equation*}
where $\mathbb R^3=\text{span}_{\mathbb R}\{I_1,I_2,I_3\}$ and
$\mathbb R^4=\text{span}_{\mathbb R}\{I_4,I_5,I_6,I_7\}$. Since
$\psi:\tilde{\Sigma}\to\mathbb R^3$ is minimal,
$\tilde{\Sigma}\times \mathbb R^4$ is the complex manifold. If
$g:\tilde{\Sigma}\to\tilde{\Sigma}$ denotes a covering
transformation, then $\psi(gp)=\psi(p)+t_g$ for some vector
$t_g\in\mathbb R^3$. It follows that the complex structure on
$\tilde{\Sigma}\times\mathbb R^4$ is invariant by the covering group
of $\Sigma$ and so descends to $\Sigma\times \mathbb R^4$. On the
other hand, for $\mathbb R^4$, we can further divide by a lattice
$\Lambda$ of translation of $\mathbb R^4$, and thereby produce a
compact complex manifold $X_\Lambda=\Sigma\times T^4$. We can view
$X_\Lambda$ as a family of complex tori, parameterized by the
Riemann surface.

Calabi showed that such complex manifolds $X_\Lambda$ are
non-K\"ahler. However, there exists a balanced metric on these
manifolds \cite{Gr,Mic}. Let us consider the {\sl nature} metric.

Define a 2-form on $X_\Lambda$ as
\begin{equation*}
\omega_0(V,W)=N\cdot(V\times W)
\end{equation*}
for any $V,W\in T_xX_\Lambda$ at any  $x\in X_\Lambda$. Then clearly
we have
\begin{equation*}
\omega_0(V,W)=-\omega_0(W,V);
\end{equation*}
and  using the formula
\begin{equation*}
N\cdot(V\times W)=(N\times V)\cdot W,
\end{equation*}
we also have
\begin{equation*}
\begin{aligned}
&\omega_0(JV,JW)=\omega_0(V,W);\\
&\omega_0(V,JV)=(N\times V)\cdot (N\times V)>0,\ \ \text{if} \ \
V\not=0.
\end{aligned}
\end{equation*}
So $\omega_0$ is the positive $(1,1)$-form on $X_\Lambda$ and
therefore defines a hermitian metric.

Next we check that  $\omega_0$ is a balanced metric.
The unit normal vector field of $X$ in $\mathbb{R}^7$
can be written as
\begin{equation}\lab{nc}
N=\sum_{j=1}^3 a_j I_j, \qquad \sum_{j=1}^3 a_j^2=1,
\end{equation}
where $a_j$ for $j=1,2,3$ are  functions on $\Sigma$.
Let $(x_4,x_5,x_6,x_7)$ be the coordinates of
$\mathbb R^4$. Then we can write the hermitian metric $\omega_0$ as
\begin{equation*}
\omega_0=\omega_{\Sigma}+\varphi_0,
\end{equation*}
where $\omega_{\Sigma}$ is a K\"ahler metric on $\Sigma$ and
\begin{equation*}
\begin{aligned}
\varphi_0=&a_1 dx_4\wedge dx_5+a_2dx_4\wedge dx_6-a_3dx_4\wedge
dx_7\\
-&a_3 dx_5\wedge dx_6 -a_2dx_5\wedge dx_7+a_1dx_6\wedge dx_7.
\end{aligned}
\end{equation*}
By direct check, we have
\begin{equation*}
\varphi_0^2=2dx_4\wedge dx_5\wedge dx_6\wedge dx_7.
\end{equation*}
Therefore,
\begin{equation*}
d(\omega^2_0)=d(2\omega_{\Sigma}\wedge
\varphi_0+\varphi_0^2)=2d\omega_{\Sigma}\wedge
\varphi_0+2\omega_{\Sigma}\wedge d\varphi_0=0,
\end{equation*}
since $\omega_{\Sigma}$ is a K\"ahler metric and all functions $a_j$
are defined on $\Sigma$.

\vspace{2mm} At last we prove that there exists a 1-Gauduchon metric
on $X_{\Lambda}$. By direct computation, we have
\begin{equation*}
\partial\bar\partial\omega_0\wedge\omega_0=\partial\bar\partial\varphi_0\wedge\varphi_0=2\sum_{j=1}^3a_j\partial\bar\partial
a_j\wedge dx_4\wedge dx_5\wedge dx_6\wedge dx_7.
\end{equation*}
Condition (\ref{nc}) implies
\begin{equation*}
\sum_{j=1}^3a_j\partial\bar\partial  a_j=-\sum_{j=1}^3\partial
a_j\wedge\bar\partial a_j,
\end{equation*}
Combining the above two equalities yields
\begin{equation*}
  \begin{split}
\sqrt{-1}\partial\bar\partial
\omega_0\wedge\omega_0
   & =-2\sqrt{-1}\sum_{j=1}^3
\partial a_j\wedge\bar\partial a_j\wedge dx_4\wedge dx_5\wedge
dx_6\wedge dx_7 \\
   &=-4\sum_{j=1}^3|\partial a_j|^2\omega^3_0,
   \end{split}
\end{equation*}
and therefore,
\begin{equation*}
\sqrt{-1}\int_{X_{\Lambda}}\partial\bar\partial(e^v\omega_0)\wedge\omega_0
=\sqrt{-1}\int_{X_{\Lambda}}e^v\omega_0\wedge\partial\bar\partial\omega_0<0.
\end{equation*}
Hence, we have $\gamma_1(\omega_0) <0$, by Corollary~\ref{co:sl2};
so $-1\in\Xi_1(X_{\Lambda})$.

\begin{prop}
$\Xi_1(X_{\Lambda})=\{-1,0,1\}$.
\end{prop}
\begin{proof}
We have proven $-1\in\Xi_1(X_\Lambda)$ and According to Theorem
\ref{th:posgam} we also have $1\in\Xi_1(X_\Lambda)$. Then by Corollary
\ref{cor0}, $0\in \Xi_1(X_{\Lambda})$.
\end{proof}

\begin{coro}
There exists a 1-Gauduchon metric on $X_{\Lambda}$.
\end{coro}


\section{The first Gauduchon metric on $S^5\times S^1$}
Let $S^5\to \mathbb P^2$ be the hopf fibration of the complex projective plane $\mathbb P^2$. $S^5$ can be viewed as the circle bundle over $\mathbb P^2$ twisted by $\frac{\omega_{FS}}{2\pi}\in H^2(\mathbb P^2,\mathbb Z)$. Here $\omega_{FS}$ is the Fubini-Study metric on $\mathbb P^2$. We let $\pi:S^5\times S^1\to \mathbb P^2$ be the natural projection. Then using a canonical way (c.f. \cite{FY,GP}), we can define  a complex structure on $S^5\times S^1$ such that $\pi$ is a holomorphic map. We can define a natural hermitian metric on $S^5\times S^1$ as follows:
\begin{equation}
\omega_0=\pi^\ast\omega_{FS}+\frac{\sqrt{-1}}2 \theta\wedge\bar \theta,
\end{equation}
where $\theta=\theta_1+\sqrt{-1}\theta_2$ is a $(1,0)$-form on $S^5\times S^1$ such that $d\theta_1=\pi^\ast\omega_{FS}$ and $d\theta_2=0$. So  $\bar\partial\theta=\pi^\ast\omega_{FS}$ and $\partial\theta=0$ which imply
\begin{equation}
\frac{\sqrt{-1}}2 \partial\bar\partial\omega_0=-\frac 1 4 \pi^\ast\omega_{FS}^2.
\end{equation}
Thus
\begin{equation}
\frac{\sqrt{-1}}2\partial\bar\partial \omega_0\wedge \omega_0=\Bigl(\frac{\sqrt{-1}}2\Bigr)^3\pi^\ast\omega_{FS}^2\wedge\theta\wedge\bar\theta
=-\frac{\omega_0^3}{3!}
\end{equation}
and therefore
$$\sqrt{-1}\int_{S^5\times S^1} \partial\bar\partial(e^v\omega_0)\wedge \omega_0=\sqrt{-1}\int_{S^5\times S^1} e^v\omega_0\wedge\partial\bar\partial\omega_0<0.$$
Hence, we have $\gamma_1(\omega_0) <0$, by Corollary~\ref{co:sl2};
so $-1\in\Xi_1(S^5\times S^1)$. Then by Corollary
\ref{cor0}, $0\in \Xi_1(S^5\times S^1)$. That is we have
\begin{prop}
There exists  a 1-Gauduchou metric on $S^5\times S^1$.
\end{prop}
Using above natural metric $\omega_0$ on $S^5\times S^1$, we can also prove
\begin{prop}
There does not exist any pluri-closed metric on $S^5\times S^1$.
\end{prop}
\begin{proof}
If there would exist a pluri-closed metric $\omega$ on $S^5\times S^1$, then
\begin{equation}
0=\int_{S^5\times S^1}\frac {\sqrt{-1}} 2\partial\bar\partial \omega\wedge\omega_0=-\frac 1 4\int_{S^5\times S^1}\omega\wedge\pi^\ast\omega_{FS}^2<0
\end{equation}
since $\omega\wedge\pi^\ast\omega_{FS}^2$ is the strictly positive definite $(3,3)$-form on $S^5\times S^1$. That is a contradiction.
\end{proof}

We also know that there does not exist any balanced metric on $S^5\times S^1$. The proof is standard and is given here.
There is an obstruction to the existence of a balanced metric on a compact complex manifold. Namely, on a compact complex manifold with a balanced metric no compact complex submanifold of codimension 1 can be homologous to 0 \cite{Mic}. Now for $\pi:S^5\times S^1\to \mathbb P^2$, since $\pi$ is a holomorphic, $\pi^{-1}(\mathbb P^2)$ is a complex hypersurface in $S^5\times S^1$. Certainly $\pi^{-1}(\mathbb P^2)$ is homologous to zero in $S^5\times S^1$ since $H^4(S^5\times S^1,\mathbb R)=0$. Therefor there  exist no balanced metric on $S^5\times S^1$.


\begin{thebibliography}{999}

\bibitem{AA} L. Alessandrini and M. Andreatta, \emph{Closed transverse
$(p,p)$-forms on compact complex manifolds}, Compositio Math. {\bf
61}(1987),  181--200. 

\bibitem{Bis} J.-M. Bismut, \emph{A local index theorem for non-K\"ahler
manifolds}, Math. Ann. {\bf 284}(1989),  681--699. 

\bibitem{Bo1} Y. Bozhkov, \textit{The specific Hermitian geometry of certain three-folds},
Riv. Mat. Univ. Parma {\bf 4}(1995), 61--68.

\bibitem{Bo2} Y. Bozhkov, \textit{The geometry of certain three-folds},
Rend. Ist. Mat. Univ. Trieste {\bf 26}(1994), 79--93.

\bibitem{Ca}E. Calabi, \emph{Construction and properties of some
6-dimensional almost complex manifolds}, Trans. Amer. Math. Soc.
{\bf 87}(1958), 407-438.


\bibitem{Dem} J.-P. Demailly, \textit{Sur l'identit\'{e} de Bochner-Kodaira-Nakano en g\'eom\'etrie hermitienne},
 88--97, Lecture Notes in Math., \textbf{1198}, Springer, Berlin, 1986. 


\bibitem{FT} A. Fino and A. Tomassini, \emph{A survey on strong KT
structures}, Bull. Math. Soc. Sci. Math. Roumanie Tome {\bf
52}(2009), 99-116.

\bibitem{FLY}J.-X. Fu, J. Li and S.-T. Yau, \emph{Balanced metrics on non-K\"ahler Calabi--Yau
threefolds}, arXiv:0809.4748. 

\bibitem{FuWangWu}
J.-X. Fu, Z. Wang and D. Wu, \emph{Form-type Calabi--Yau equations},
Math. Res. Lett. \textbf{17}(2010), 887--903.

\bibitem{FY} J.-X. Fu and S.-T. Yau, \textit{The theory of
superstring with flux on non-K\"ahler manifolds and the complex
Monge-Amp\`ere equation}, J. Differential Geom. {\bf 78}(2008), 369-428.

\bibitem{Gau}
P. Gauduchon, \emph{Sur la 1-forme de torsion d'une
vari$\acute{e}$t$\acute{e}$ hermitienne compacte}, Math. Ann.
\textbf{267}(1984), 495--518.

\bibitem{GP} P. Goldstein and S. Prokushkin, \emph{Geometric model
for complex non-K\"ahler manifolds with $SU(3)$ structures}, Commun.
Math. Phys. 251(2004), 65-78.

\bibitem{GGP} D. Grantcharov, G. Grantcharov and Y.-S. Poon,
\emph{Calabi-Yau connections with torsion on tori bundles}, J. Diff.
Geom. {\bf 78}(2008), 13-32.

\bibitem{Gr} A. Gray, \emph{Some examples of almost Hermitian
manifolds}, Illinois J. Math. {\bf 10}(1966), 353-366.


\bibitem{GH}A. Gray and L. M. Hervella, \emph{The sixteen classes of
almost Hermitian manifolds and their linear invariants}, Annali Mat.
Pura Appl. {\bf 123}(1980), 35--58.

\bibitem{GrHa}
P. Griffiths and J. Harris, \emph{Principles of Algebraic Geometry}, John Wiley \& Sons, 1978.


\bibitem{JY}J. Jost and S.-T. Yau, \emph{A nonlinear
elliptic system for maps from Hermitian to Riemannian manifolds and
rigidity theorems in Hermitian geometry}, Acta Math. {\bf
170}(1993), 221--254. 

\bibitem{LiYau}
J. Li and S.-T. Yau, \emph{Hermitian--Yang--Mills connection on
non-K\"ahler manifolds}, Mathematical aspects of string theory (San
Diego, Calif., 1986),  560--573, Adv. Ser. Math. Phys., \textbf{1},
World Sci. Publishing, Singapore, 1987.



\bibitem{Mic}
M. L. Michelsohn, \textit{On the existence of special metrics in
complex geometry}, Acta Math. {\bf 149}(1982),  261--295. 

\bibitem{MK}
J. Morrow and K. Kodaira, \emph{Complex Manifolds}, Holt, Rinehart and Winston, 1971.

\bibitem{ST}
J. Streets and G. Tian, \emph{A parabolic flow of pluriclosed metrics},
arXiv:0903.4418.



\end{thebibliography}
\end{document}